\documentclass[smallextended,numbook,runningheads]{svjour3}     
\smartqed  
\usepackage{graphicx}
\usepackage{amsmath}
\usepackage{epstopdf} 
\usepackage{mathrsfs}
\usepackage{mathabx}
\usepackage{cases}
\usepackage{color}

 \usepackage{diagbox}

\newtheorem{assumption}{Assumption}

\let \mr=\mathrm
\let \b=\boldsymbol
\begin{document}

\title{Convergence of a finite element method  on  a Bakhvalov-type mesh   for a singularly perturbed convection--diffusion equation in 2D\thanks{This research is partially supported by National Natural Science Foundation of China (11771257,11601251).
}}

\titlerunning{Finite element method on Bakhvalov-type mesh}        

\author{Jin Zhang         \and
        Xiaowei Liu 
}


\institute{Jin Zhang \at
              School of Mathematics and Statistics, Shandong Normal University,
Jinan 250014, China\\
              \email{jinzhangalex@hotmail.com}           
           \and
           Corresponding author:  Xiaowei Liu \at
            School of Mathematics and Statistics, Qilu University of Technology (Shandong Academy of Sciences), Jinan 250353, China\\
            \email{xwliuvivi@hotmail.com}  
}

\date{Received: date / Accepted: date}

\maketitle

\begin{abstract}
A  finite element method of any order is applied  on a Bakhvalov-type mesh to solve a singularly perturbed convection--diffusion equation in 2D, whose solution exhibits exponential boundary layers.  
A uniform convergence of (almost) optimal order is proved by means of a carefully defined interpolant. 
\end{abstract}
\section{Introduction}
Consider the elliptic boundary value problem 
 \begin{equation}\label{eq:model problem}
\begin{split}
-\varepsilon\Delta u-\boldsymbol{b}\cdot \nabla u+cu=&f\quad\text{
in $\Omega=(0,1)^{2}$},\\
 u=&0 \quad \text{on $\partial\Omega$},
 \end{split}
 \end{equation}
where $\varepsilon\ll 1$ is a small positive parameter and $\boldsymbol{b}(x,y)=(b_{1}(x,y),b_{2}(x,y))^{T}$. The functions $b_{1}, b_{2}, c$ and $f$ are assumed to be smooth on~$\bar\Omega$.
We also assume  that
\begin{equation}\label{betabounds}
b_{1}(x,y)\ge \beta_{1}>0,\,b_{2}(x,y)\ge \beta_{2}>0, \, c(x,y)+\frac{1}{2}\nabla\cdot \b{b}(x,y)\ge \gamma>0\quad \text{on $\bar{\Omega},$}
\end{equation}
where $\beta_{1}$, $\beta_{2}$ and $\gamma$ are some constants. 
These conditions ensure that~\eqref{eq:model problem} has a unique solution
in $H^1_0(\Omega)\cap H^2(\Omega)$ for all $f\in L^2(\Omega)$  (see, e.g., \cite{Roo1Sty2Tob3:2008-Robust}). Because $\varepsilon$ is small, the problem is in general singularly perturbed and its solution  typically has exponential boundary layers
at $x = 0$ and $y = 0$ and a corner layer at $(0, 0)$. 

Layer phenomena appears in different kinds of problems, for example singularly perturbed problems, which is one of the important topics in scientific computing. If a priori knowledge of layers has been obtained from asymptotic analysis etc.,  different kinds of meshes could be designed  for uniform convergent numerical methods (see \cite{Roo1Sty2Tob3:2008-Robust,Mil1Rio2Shi3:2012-Fitted}).  Here ``uniform'' means that the convergence is independent of the singular perturbation parameter. Bakhvalov-type meshes are one of the most popular layer-adapted meshes and usually have better numerical performances  than Shishkin-type meshes---another popular layer-adapted meshes ( see \cite{Linb:2010-Layer}). 


However, it is far from mature for convergence theories of finite element methods on Bakhvalov-type meshes. 
One of the main reasons is that the standard Lagrange interpolant does not work for Bakhvalov-type meshes (see  \cite{Roos-2006-Error}). In \cite{Zhan1Liu2:2020-Optimal}, we gave a new idea for  convergence analysis on Bakhvalov-type meshes in 1D. Here we extend the analysis
to two dimensions.  The extension is not trivial, because we must pay attention to the construction of  the interpolant used in the case of 2D. This interpolant is carefully defined according to the characteristics of layer functions and the structures of Bakhvalov-type meshes. Besides, different from 1D case, we must pay attention to the homogenous Dirichlet boundary condition when the idea in  \cite{Zhan1Liu2:2020-Optimal} is applied to 2D case.  The interpolation errors are derived in a delicate way. Then   almost uniform convergence of optimal order  is proved for finite element methods. 

The rest of the paper is organized as follows. In Section 2 we describe the   assumptions
on the regularity of the solution, introduce a Bakhvalov-type mesh    and define a finite element method of any order.  Some  preliminary results for the subsequent analysis are also given in this section.  In Section 3 we construct and analyze an interpolant  to the solution  for   uniform convergence on  the Bakhvalov-type mesh. In Section 4 almost uniform convergence of optimal order is obtained by means of the interpolant  and careful analysis of the convective term in the bilinear form. 
In Section 5, numerical
results illustrate our theoretical bounds.

Denote by $\Vert \cdot \Vert_{ \infty,D }$  the norms in the Lebesgue space $L^{\infty}(D)$.   In $L^2(D)$, the inner product and the $L^2(D)$-norm are denoted by  $(\cdot,\cdot)_D$ and $\Vert \cdot \Vert_{D}$, respectively. In $H^1(D)$, the seminorms are denoted by $\vert \cdot \vert_1$. Here $D$ is any measurable subset of $\Omega$. When $D=\Omega$ we drop the subscript $D$ from the  notation for simplicity.
Throughout the article, all constants $C$   are independent of
$\varepsilon$ and the mesh parameter $N$ and may take different values in different formulas. 

%
%
%
%
\section{Decomposition of the solution, Bakhvalov-type mesh and finite element method}\label{sec:decomp} 
In this section we present a decomposition of the solution to \eqref{eq:model problem}, introduce  a Bakhvalov-type mesh  and define a finite element method. Some preliminary inequalities are also presented. In the subsequent analysis, let $k$ be a fixed integer with $k\ge 1$.

\subsection{Regularity of the solution}
We make the following assumption about the solution  $u$ to \eqref{eq:model problem}, which describes the layer structure of~$u$. This assumption is also used in~\cite{Fra1Lin2Roo3etc:2010-Uniform}.
 
\begin{assumption}\label{assumption:1}
The solution $u$ of \eqref{eq:model problem} can be decomposed as
\begin{subequations}
\begin{equation}\label{eq:decomposition}
u=S+E_{1}+E_{2}+E_{12},
\end{equation}
where $S$ is the smooth part of $u$, $E_1$  and $E_2$ are exponential layers along
the sides $x=0$ and $y=0$ of $\Omega$ respectively, while $E_{12}$ is an exponential corner layer at $(0,0)$.
Moreover,there exists a  constant $C$ such that for all $(x,y)\in \bar{\Omega}$ and $0\le i+j \le k+1$ one has
\begin{align}
\left|\frac{\partial^{i+j}S}{\partial x^{i}\partial y^{j}}(x,y) \right| &\le C, \label{eq:(2.1b)}\\
\left|\frac{\partial^{i+j}E_{1}}{\partial x^{i}\partial y^{j}}(x,y) \right| &\le C\varepsilon^{-i}e^{-\beta_{1}x/\varepsilon},\label{eq:(2.1c)}\\
\left|\frac{\partial^{i+j}E_{2}}{\partial x^{i}\partial y^{j}}(x,y) \right| &\le C\varepsilon^{-j}e^{-\beta_{2}y/\varepsilon},\nonumber\\
\left|\frac{\partial^{i+j}E_{12}}{\partial x^{i}\partial y^{j}}(x,y) \right| &\le C\varepsilon^{-(i+j)}e^{-(\beta_{1}x+\beta_{2}y)/\varepsilon}.\label{eq:(2.1e)}
\end{align}
\end{subequations}
\end{assumption}

\begin{remark}
For the case $k=1$, the existence of this decomposition of $u$ with the bounds on derivatives can be guaranteed by conditions on the data of the problem \eqref{eq:model problem} (see \cite{Linb1Styn2:2001-Asymptotic}). The arguments in \cite{Linb1Styn2:2001-Asymptotic} make this assumption with $k\ge 2$ credible if we impose  sufficient compatibility conditions on $f$ (see some explanations in \cite[Sect.7]{Stynes:2005-Steady}).
\end{remark}

\subsection{Bakhvalov-type mesh}
Bakhvalov mesh  first appeared in \cite{Bakhvalov:1969-Towards} and is graded in the layer. Its  applications require the solution of a nonlinear equation. To avoid this difficulty,  Bakhvalov-type meshes are proposed as approximations of  Bakhvalov mesh (see \cite{Linb:2010-Layer}).

Let $N$ be an even positive integer. We introduce a Bakhvalov-type mesh in the $x$-direction
$$
0=x_0<x_1<\ldots<x_{N-1}<x_N=1.
$$
To resolve the layer along $x=0$, the mesh is graded in $[x_0,x_{N/2}]$ and equidistant
in $[x_{N/2},1]$. The mesh points $x_i$ is defined by
\begin{equation}\label{B-mesh-1}
x_i=
\left\{
\begin{aligned}
&\frac{\sigma\varepsilon}{\beta_1}\varphi(t_i)\text{ with $t_i=i/N$}\quad
&&\text{for $i=0,1,\ldots,N/2$},\\
&1-(1-x_{N/2})2(N-i)/N \quad
&&\text{for $i=N/2+1,\ldots,N$},
\end{aligned}
\right.
\end{equation}
with $\sigma\ge k+1$ and $\varphi(t):=-\ln(1-2(1-\varepsilon)t)$. The parameter  $\sigma$  determines the smallness of the layer terms in $x_{N/2}$. This Bakhvalov-type mesh  is also analyzed in \cite{Roos-2006-Error}. 
In a similar way we define the mesh $\{ y_j \}_{j=0}^N$ along the $y$-direction except that we replace $\beta_1$ by $\beta_2$ in \eqref{B-mesh-1}. Then we obtain a tensor-product rectangular mesh $\mathcal{T}_N$ with mesh points $(x_i,y_j)$ (see Figure \ref{fig:B-mesh}).

\begin{assumption}\label{assumption-2}
Assume  that $\varepsilon\le 
N^{-1}$    in our analysis, as is not a restriction in practice. 
\end{assumption} 
 
\begin{figure}
\centering
\includegraphics[width=2.8in]{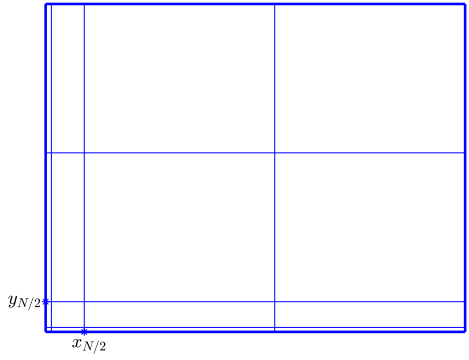}
\caption{Bakhvalov-type mesh $\mathcal{T}_{N}$.}
\label{fig:B-mesh}
\end{figure}
 
Moreover, assume 
$N\ge 4$ is an even integer. Set $h_{i,x}:=x_{i+1}-x_{i}$ and $h_{j,y}:=y_{j+1}-y_{j}$ for all~$i,j$. A mesh rectangle is often written as $\tau_{i,j}=[x_{i},x_{i+1}]\times [y_{j},y_{j+1}]$ for a specific element and more simply as $\tau$ for a generic mesh rectangle.

According to \cite[Lemma 3]{Zhan1Liu2:2020-Optimal}, we have the following  lemma.
\begin{lemma}\label{lem:Bakhvlov-mesh}
For  the Bakhvalov-type mesh  \eqref{B-mesh-1}, one has
\begin{align}
& h_{0,x}\le h_{1,x}\le \ldots \le h_{N/2-2,x},\label{eq:mesh-1}\\
&C\varepsilon N^{-1}\le h_{0,x}\le C\varepsilon N^{-1},\label{eq:mesh-2}\\
&\frac{1}{4}\sigma \varepsilon \le h_{N/2-2,x}\le \sigma \varepsilon,\label{eq:mesh-3}\\
&\frac{1}{2}\sigma \varepsilon \le h_{N/2-1,x}\le 2\sigma N^{-1},\label{eq:mesh-4}\\
&N^{-1}\le h_{i,x}\le 2N^{-1}\quad  N/2\le i \le N-1,\label{eq:mesh-5}\\
&x_{N/2-1}\ge C\sigma \varepsilon \ln N,\quad x_{N/2}\ge C\sigma \varepsilon |\ln \varepsilon |\label{eq:mesh-6},\\
 &h_{i,x}^{ \mu} e^{- \beta_1 x_i/\varepsilon}\le C\varepsilon^{\mu} N^{-\mu} 
\quad \text{for $0\le i\le N/2-2$  and $0\le \mu\le \sigma$}.
\label{eq:mesh-7}
\end{align}
For $h_{j,y}$, $0\le j <N$,  bounds   analogous to \eqref{eq:mesh-1}--\eqref{eq:mesh-7} also hold.
\end{lemma}
\subsection{Finite element method}
On the above Bakhvalov-type mesh, define the finite element space by 
\begin{equation*}\label{eq:VN}
V^{N}:=\{v^{N}\in C(\bar{\Omega}):\;\;v^{N}|_{\partial\Omega}=0
 \text{ and $v^{N}|_{\tau}\in \mathcal{Q}_k(\tau)$   }
  \;\forall\tau\in\mathcal{T}_{N} \},
\end{equation*}
where $\mathcal{Q}_k(\tau)=\mr{span}\{ x^i y^j:\; 0\le i,j\le k\}$.

The finite element method is defined as follows: Find $u^N\in V^{N}$ such that
\begin{equation}\label{eq: FE}
a(u^N,v^{N})=(f,v^{N})\quad \forall v^{N}\in V^{N},
\end{equation}
with
$
a(u^N,v^{N}):= \varepsilon(\nabla u^N,\nabla v^{N})+(-\boldsymbol{b}\cdot \nabla u^N+c u^N,v^{N}).
$
Condition \eqref{betabounds} implies the coercivity
\begin{equation}\label{eq:coercivity}
a(v^N,v^N) \ge \alpha \Vert v^N \Vert_{\varepsilon}^2\quad \text{for all $v^N\in V^N$},
\end{equation}
where  $\alpha=\min \{1,\gamma\}$ and
\begin{equation*}
\Vert v \Vert_{ \varepsilon}:=
\left\{\varepsilon \vert  v \vert^2_1+ \Vert v \Vert^2\right\}^{1/2} \quad \forall v\in H^1(\Omega).
\end{equation*} 
It follows that  there exists a unique solution $u^N$ for problem \eqref{eq: FE} from Lax-Milgram lemma.
Clearly, \eqref{eq:model problem} and \eqref{eq: FE} imply the Galerkin orthogonality property
\begin{equation}\label{eq:Galerkin orthogonality}
a(u-u^N,v^N)=0\quad \text{for all $v^N\in V^N$}.
\end{equation}


\section{Interpolation and its errors}
A new interpolation operator  is introduced for   uniform convergence. Set $x_{i}^{s}:=x_{i}+(s/k)h_{i,x}$ and $y_{j}^{t}:=y_{j}+(t/k)h_{j,y}$  for $i,j=0,1,\ldots,N-1$ and $s,t=0,\ldots,k-1$. For the consistency of notation, set $x_{N}^{0}=x_N$ and $y_{N}^{0}=y_N$.
  For any $v\in C^0(\bar{\Omega})$ its standard Lagrange interpolant $v^I\in V^{N}$ on the Bakhvalov-type mesh   can be written in the following form
\begin{align*}
v^I(x,y)=&\sum_{i=0}^{N-1} \sum_{s=0}^{k-1}\left( 
\sum_{j=0}^{N-1} \sum_{t=0}^{k-1} v(x_{i}^{s},y_{j}^{t})\theta_{i,j}^{s,t}(x,y)+v(x_{i}^{s},y_{N}^{0})\theta_{i,N}^{s,0}(x,y)\right)\\
&+\sum_{j=0}^{N-1} \sum_{t=0}^{k-1} v(x_{N}^{0},y_{j}^{t})\theta_{N,j}^{0,t}(x,y)+v(x_{N}^{0},y_{N}^{0})\theta_{N,N}^{0,0}(x,y),
\end{align*}
where  $\theta_{i,j}^{s,t}(x,y)\in V^{N}$ is the piecewise $k$th--order  hat function associated with the point  $(x_{i}^{s},y_{j}^{t})$.  For the solution $u$ to \eqref{eq:model problem}, recall \eqref{eq:decomposition} in Assumption \ref{assumption:1} and define the interpolant $\Pi u$ 
 by
\begin{equation}\label{eq:interpolant-u}
\Pi u=S^I+\pi_1 E_1+\pi_2 E_2+\pi_{12} E_{12}.
\end{equation}
Here $S^I$ is the Lagrange interpolant to $S$ and 
\begin{equation}\label{eq:interpolant-E-1}
\begin{aligned}
(\pi_i E_i)(x,y)=&E^I_i-\mathcal{P}_i E_i -\mathcal{B}_iE_i\quad\text{for $i=1,2$},\\
(\pi_{12} E_{12})(x,y)=&E^I_{12}-\mathcal{P}_{12} E_{12}
\end{aligned}
\end{equation}
where  
\begin{align}
\mathcal{P}_1 E_1=\sum_{i=N/2-1}\sum_{s=0}^{k-1}\left( 
\sum_{j=0}^{N-1} \sum_{t=0}^{k-1} E_1(x_{i}^s,y_j^t)\theta_{i,j}^{s,t}(x,y)+E_1(x_{i}^{s},y_{N}^{0})\theta_{i,N}^{s,0}(x,y)\right),\label{eq:P1E1}\\
\mathcal{P}_2 E_2=\sum_{j=N/2-1}\sum_{t=0}^{k-1}\left( 
\sum_{i=0}^{N-1} \sum_{s=0}^{k-1} E_2(x_{i}^{s},y_{j}^{t})\theta_{i,j}^{s,t}(x,y)+E_2(x_{N}^{0},y_{j}^{t})\theta_{N,j}^{0,t}(x,y)\right)\label{eq:P2E2},\\
\mathcal{P}_{12} E_{12}=\sum_{i=N/2-1}\sum_{j=N/2-1}\left(\sum_{s=0}^{k-1}\sum_{t=0}^{k-1} E_{12}(x_{i}^s,y_j^t)\theta_{i,j}^{s,t}(x,y) \right)\label{eq:P12E12}
\end{align}
and 
\begin{equation}\label{eq:B1+B2}
\begin{aligned}
\mathcal{B}_1E_1=\sum_{i=N/2-1}\sum_{s=0}^{k-1}\sum_{j=0,N}E_1(x_{i}^s,y_j^0)\theta_{i,j}^{s,0}(x,y),\\
\mathcal{B}_2E_2=\sum_{i=0,N}\sum_{j=N/2-1}\sum_{t=0}^{k-1}E_2(x_{i}^0,y_j^t)\theta_{i,j}^{0,t}(x,y).\\
\end{aligned}
\end{equation}
Clearly  we have
\begin{align}
&\Pi u\in V^N,\quad \Pi u=u^I-\sum_{i=1,2,12}\mathcal{P}_i E_i.\label{eq:pi-E-0} 
\end{align}
\begin{remark}
The definitions of interpolation operators $\pi_i$, $i=1,2,12$ arise  from layer functions  $E_i$, $i=1,2,12$ and the construction of Bakhvalov-type mesh \eqref{B-mesh-1}. 
The idea for operators $\mathcal{P}_1$, $\mathcal{P}_2$ and $\mathcal{P}_{12}$ is same to the operator $\mathcal{P}$ in \cite{Zhan1Liu2:2020-Optimal}, that is, new interpolations for layer functions  are zero at certain degrees of freedom. The operators $\mathcal{B}_1$ and $\mathcal{B}_2$ are introduced in order to maintain  homogenous Dirichlet boundary conditions.

\end{remark}

 From \cite[Theorem 2.7]{Apel:1999-Anisotropic}, we have the following anisotropic interpolation results.
 \begin{lemma}\label{lem:anisotropic-interpolation}
 Let $\tau\in\mathcal{T}_N$ and $v\in H^{k+1}(\tau)$.Then there exists a constant $C$ such that Lagrange interpolation $v^I$ satisfies
 \begin{align*}
\Vert  v-v^I \Vert_{\tau}\le & C\sum_{i+j=k+1}h^i_{x,\tau}h^j_{y,\tau} 
\left\Vert \frac{\partial^{k+1} v }{ \partial x^{i} \partial y^{j}  } \right\Vert_{\tau}\\
\Vert  (v-v^I)_x \Vert_{\tau}\le & C\sum_{i+j=k}h^i_{x,\tau}h^j_{y,\tau} 
\left\Vert \frac{\partial^{k+1} v }{ \partial x^{i+1} \partial y^{j}  } \right\Vert_{\tau}
\end{align*} 
and
$$
\Vert  (v-v^I)_y \Vert_{\tau}\le   C\sum_{i+j=k}h^i_{x,\tau}h^j_{y,\tau} 
\left\Vert \frac{\partial^{k+1} v }{ \partial x^{i} \partial y^{j+1}  } \right\Vert_{\tau},
$$
where $h_{x,\tau}$ and $h_{y,\tau}$ denote the lengths along $x-$axis and $y-$axis of the rectangle $\tau$, respectively.
\end{lemma}

%
%
%

%
%
%
%
Set
\begin{align*}
\Omega_{11}:=[x_0,x_{N/2-1}]\times [y_0,y_{N/2-1}], \quad
\Omega_{12}:=[x_{N/2-1},x_N]\times [y_0,y_{N/2-1}],\\ 
\Omega_{21}:=[x_0,x_{N/2-1}]\times [y_{N/2-1},y_N], \quad
\Omega_{22}:=[x_{N/2-1},x_N]\times [y_{N/2-1},y_N].
\end{align*}
\begin{lemma}\label{lem:interpolation-error-1}
Let Assumptions \ref{assumption:1} and \ref{assumption-2} hold true.
Let $E^I_i$, $i=1,2$, denote the Lagrange interpolants of $
E_i$, $i=1,2$, respectively, on the Bakhvalov-type mesh $\mathcal{T}_N$. Then there exists a constant $C$ such
that the following interpolation error estimates hold true: 
\begin{align*}
&\Vert E_1-E^I_1 \Vert_{\Omega \setminus \big( (x_{N/2-1},x_{N/2})\times [0,1] \big) }
+
\Vert E_2-E^I_2 \Vert_{\Omega \setminus \big( [0,1]\times (y_{N/2-1},y_{N/2}) \big) }
\le C \varepsilon^{1/2}N^{-(k+1/2)}, \\ 
&\Vert E_1-E^I_1 \Vert+\Vert E_2-E^I_2 \Vert\le  CN^{-(k+1)},\\
&\Vert E_1-E^I_1 \Vert_{\varepsilon}+\Vert E_2-E^I_2 \Vert_{\varepsilon}\le C N^{-k},  \\
&\Vert  \mathcal{P}_1 E_1   \Vert_{\varepsilon}+\Vert  \mathcal{P}_2 E_2 \Vert_{\varepsilon}+\Vert  \mathcal{B}_1 E_1   \Vert_{\varepsilon}+\Vert  \mathcal{B}_2 E_2 \Vert_{\varepsilon}  \le CN^{1/2-\sigma}, 
\end{align*}
where $ \mathcal{P}_i E_i$, $i=1,2$, are defined in \eqref{eq:P1E1} and \eqref{eq:P2E2}, respectively.
\end{lemma}
\begin{proof}
We just consider $E_1$, since $E_2$ can be analyzed in a similar way. To consider $\Vert E_1-E^I_1 \Vert$, we decompose it as follows
\begin{equation*}
\begin{aligned}
\Vert E_1-E^I_1 \Vert^2=&\Vert E_1-E^I_1 \Vert^2_{[x_0,x_{N/2-1}]\times [0,1]}
+\Vert E_1-E^I_1 \Vert^2_{[ x_{N/2-1},x_{N/2}]\times [0,1]}\\
&+\Vert E_1-E^I_1 \Vert^2_{[ x_{N/2},x_{N}]\times [0,1]}\\
=:&A_1+A_2+A_3.
\end{aligned}
\end{equation*}

Note $h_{i,x}\le C \varepsilon$ for $i=0,\ldots,N/2-2$. Lemmas  \ref{lem:Bakhvlov-mesh} and  \ref{lem:anisotropic-interpolation} yield
\begin{equation}\label{eq:A1}
\begin{aligned}
A_1=&\sum_{i=0}^{N/2-2}\sum_{j=0}^{N-1} \Vert E_1-E^I_1 \Vert^2_{\tau_{i,j}} 
\le   C \sum_{i=0}^{N/2-2}\sum_{j=0}^{N-1}\sum_{l+m=k+1}h^{2l}_{i,x}h^{2m}_{j,y}
\left\Vert \frac{ \partial^{k+1} E_1 }{ \partial x^l \partial y^m } \right \Vert^2_{\tau_{i,j}}\\
\le & C \sum_{i=0}^{N/2-2}\sum_{j=0}^{N-1}\sum_{l+m=k+1}h^{2l}_{i,x}h^{2m}_{j,y}
\; (\varepsilon^{-2l}  e^{-2\beta_1 x_i/\varepsilon}h_{i,x} h_{j,y})\\
\le & C\sum_{i=0}^{N/2-2}\sum_{j=0}^{N-1}\sum_{l+m=k+1} (\varepsilon^{2l} N^{-2 l})\;
h_{j,y}^{2m+1} \varepsilon^{1-2l}\\
\le & C \varepsilon N^{-(2k+1)}.
\end{aligned}
\end{equation}

Now we consider the term $A_2$. 
Set $D_0:=[ x_{N/2-1},x_{N/2}]\times [0,1]$. Recall $|E_1(x_{N/2-1}^i,y_j^t)|\le C N^{-\sigma}$.
Then we have
\begin{equation}\label{eq:A2-1}
\begin{aligned}
\Vert  E_1^I \Vert_{ D_0 }^2  \le & C N^{-2\sigma}
\left( \sum_{i=N/2-1}\sum_{s=0}^{k-1}\left(\sum_{j=0}^{N-1}\sum_{t=0}^{k-1} \Vert \theta_{i,j}^{s,t} \Vert_{ D_0 }^2
+  \Vert \theta_{i,N}^{s,0} \Vert_{D_0}^2 \right) \right)\\
&+C N^{-2\sigma}
\left( \sum_{i=N/2} \left(\sum_{j=0}^{N-1}\sum_{t=0}^{k-1} \Vert \theta_{i,j}^{0,t} \Vert_{ D_0 }^2
+  \Vert \theta_{i,N}^{0,0} \Vert_{D_0}^2 \right) \right)\\
\le & C N^{-2\sigma}  \sum_{j=0}^{N-1} h_{N/2-1,x} h_{j,y} 
\le C N^{-(2\sigma+1)},
\end{aligned}
\end{equation}
where we have used Lemma \ref{lem:Bakhvlov-mesh}. Direct calculations yield
\begin{equation}\label{eq:A2-2}
\Vert E_1 \Vert_{D_0}\le C \varepsilon^{1/2} N^{-\sigma}.
\end{equation}
From \eqref{eq:A2-1} and \eqref{eq:A2-2} we obtain
\begin{equation}\label{eq:A2}
A_2\le C(\varepsilon+N^{-1}) N^{-2\sigma}.
\end{equation}

Note $|E_1(x_{i}^s,y_j^t)|\le C \varepsilon^{\sigma}$ for $i\ge N/2$. Then the triangle inequality and
H\"{o}lder inequalities yield
\begin{equation}\label{eq:A3}
\begin{aligned}
A_3\le &
C\left( \Vert E_1 \Vert^2_{[ x_{N/2},x_{N}]\times [0,1]}+\Vert E_1^I \Vert^2_{[ x_{N/2},x_{N}]\times [0,1]} \right)\\
\le &
C\left( \Vert E_1 \Vert^2_{\infty,[ x_{N/2},x_{N}]\times [0,1] }+\Vert E_1^I \Vert^2_{  \infty, [ x_{N/2},x_{N}]\times [0,1]  }  \right)\\
\le &
C \varepsilon^{2\sigma}.
\end{aligned}
\end{equation}
Collecting \eqref{eq:A1}, \eqref{eq:A2} and \eqref{eq:A3}, we prove the first and the second  bounds.

Now we consider $\vert E_1-E_1^I \vert^2_1=\Vert (E_1-E_1^I)_x \Vert^2+\Vert (E_1-E_1^I)_y \Vert^2$.  
Lemmas \ref{lem:Bakhvlov-mesh} and  \ref{lem:anisotropic-interpolation} give
\begin{equation}\label{eq:B1}
\begin{aligned}
&\Vert (E_1-E_1^I)_x \Vert^2_{\Omega_{11}\cup\Omega_{21}}=\sum_{i=0}^{N/2-2}\sum_{j=0}^{N-1} \Vert (E_1-E^I_1)_x \Vert^2_{\tau_{i,j}} \\
\le  & C \sum_{i=0}^{N/2-2}\sum_{j=0}^{N-1}\sum_{l+m=k}h^{2l}_{i,x}h^{2m}_{j,y}
\left\Vert \frac{ \partial^{k+1} E_1 }{ \partial x^{l+1} \partial y^m } \right \Vert^2_{\tau_{i,j}}\\
\le & C \sum_{i=0}^{N/2-2}\sum_{j=0}^{N-1}\sum_{l+m=k}h^{2l}_{i,x}h^{2m}_{j,y}
\; (\varepsilon^{-2(l+1)}  e^{-2\beta_1 x_i/\varepsilon} h_{i,x}h_{j,y})\\
\le & C\sum_{i=0}^{N/2-2}\sum_{j=0}^{N-1}\sum_{l+m=k} (\varepsilon^{2l+1} N^{-(2 l+1)})\;
h_{j,y}^{2m+1} \varepsilon^{-2(l+1)}\\
\le & C \varepsilon^{-1} N^{-2k}.
\end{aligned}
\end{equation}
Note $\Vert (E_1)_x \Vert_{\Omega_{12}\cup\Omega_{22}}\le C\varepsilon^{-1/2}N^{-\sigma}$. 
Then from the triangle inequality one has
\begin{equation}\label{eq:B2}
\begin{aligned}
&\Vert (E_1-E_1^I)_x \Vert^2_{\Omega_{12}\cup\Omega_{22}}
\le  2 \Vert (E_1)_x \Vert^2_{\Omega_{12}\cup\Omega_{22}}+2\Vert (E_1^I)_x \Vert^2_{\Omega_{12}\cup\Omega_{22}}\\
\le & 
2 \Vert (E_1)_x \Vert^2_{\Omega_{12}\cup\Omega_{22}}+2\sum_{j=0}^{N-1} \Vert (E_1^I)_x \Vert^2_{\tau_{N/2-1,j}}+2\sum_{i=N/2}^{N-1}\sum_{j=0}^{N-1} \Vert (E_1^I)_x \Vert^2_{\tau_{i,j}}\\
\le &
C \varepsilon^{-1}N^{-2\sigma}+C\varepsilon^{2\sigma}N^2,
\end{aligned}
\end{equation}
where inverse inequalities \cite[Theorem 3.2.6]{Ciarlet:1978-finite} and Lemma \ref{lem:Bakhvlov-mesh} yield
\begin{align*}
&\sum_{j=0}^{N-1} \Vert (E_1^I)_x \Vert^2_{\tau_{N/2-1,j}}+\sum_{i=N/2}^{N-1}\sum_{j=0}^{N-1} \Vert (E_1^I)_x \Vert^2_{\tau_{i,j}}\\
\le & C \sum_{j=0}^{N-1} h^{-2}_{N/2-1,x} \Vert E_1^I \Vert^2_{\tau_{N/2-1,j}}+ C \sum_{i=N/2}^{N-1}\sum_{j=0}^{N-1} h^{-2}_{i,x} \Vert E_1^I \Vert^2_{\tau_{i,j}}\\
\le &C \sum_{j=0}^{N-1} h^{-2}_{N/2-1,x} \Vert E_1^I \Vert^2_{ \infty,\tau_{N/2-1,j}  } h_{N/2-1,x} h_{j,y}+C \sum_{i=N/2}^{N-1}\sum_{j=0}^{N-1} h^{-2}_{i,x} \Vert E_1^I \Vert^2_{ \infty,\tau_{i,j}  } h_{i,x}h_{j,y}\\
\le & C \varepsilon^{-1} N^{-2\sigma}+C \varepsilon^{2\sigma} N^{2}.
\end{align*}

Similar to the derivations of \eqref{eq:B1}, we have
\begin{equation}\label{eq:B3}
\Vert (E_1-E_1^I)_y \Vert^2_{\Omega_{11}\cup\Omega_{21}}\le C \varepsilon N^{-2k}.
\end{equation}
Note $\Vert (E_1)_y \Vert_{\Omega_{12}\cup\Omega_{22}}\le C\varepsilon^{1/2}N^{-\sigma}$ and $h_{j,y}\ge C\varepsilon N^{-1}$ for $j=0,\ldots,N-1$. Then one has
\begin{equation}\label{eq:B4}
\begin{aligned}
\Vert (E_1-E_1^I)_y \Vert^2_{\Omega_{12}\cup\Omega_{22}}
\le &  2 \Vert (E_1)_y \Vert^2_{\Omega_{12}\cup\Omega_{22}}+2\Vert (E_1^I)_y \Vert^2_{\Omega_{12}\cup\Omega_{22}} \\
\le &  C \varepsilon N^{-2\sigma}+C\varepsilon^{-1} N^{2-2\sigma},
\end{aligned}
\end{equation}
where 
\begin{align*}
\Vert (E_1^I)_y \Vert^2_{\Omega_{12}\cup\Omega_{22}}=&\sum_{i=N/2-1}^{N-1}\sum_{j=0}^{N-1} \Vert (E_1^I)_y \Vert^2_{\tau_{i,j}}
\le  C \sum_{i=N/2-1}^{N-1}\sum_{j=0}^{N-1} h^{-2}_{j,y} \Vert E_1^I \Vert^2_{\tau_{i,j}}\\
\le &C \sum_{i=N/2}^{N-1}\sum_{j=0}^{N-1} h^{-2}_{j,y}( N^{-2\sigma} h_{i,x}h_{j,y})\\
\le & C   N^{1-2\sigma} \max_j h_{j,y}^{-1}\le \varepsilon^{-1} N^{2-2\sigma}. 
\end{align*}
Collecting \eqref{eq:B1}--\eqref{eq:B4} and considering $\Vert E_1 -E_1^I \Vert\le CN^{-(k+1)}$, we prove the third bound.

Now we consider $\Vert \mathcal{P}_1 E_1 \Vert_{\varepsilon}$. From \eqref{eq:P1E1} and $|E(x_{N/2-1}^i,y_j^t)|\le CN^{-\sigma}$, we can easily obtain
\begin{equation}\label{eq:B5}
\begin{aligned}
\Vert \mathcal{P}_1 E_1 \Vert_{\varepsilon}^2
\le &
CN^{-2\sigma}   \sum_{s=0}^{k-1} \left( \sum_{j=0}^{N-1}  \sum_{t=0}^{k-1}\Vert \theta_{N/2-1,j}^{s,t} \Vert^2_{\varepsilon}+
\Vert \theta_{N/2-1,N}^{s,0} \Vert^2_{\varepsilon} \right)\\
\le &
CN^{-2\sigma}    \sum_{j=0}^{N-1}  
\left( \varepsilon h^{-1}_{N/2-1,x}h_{j,y}+\varepsilon h_{N/2-1,x}h^{-1}_{j,y} +h_{N/2-1,x}h_{j,y}\right)\\
\le &
CN^{1-2\sigma},
\end{aligned}
\end{equation}
where we have used $h_{j,y}^{-1}\le C\varepsilon^{-1} N$ from Lemma \ref{lem:Bakhvlov-mesh}.   The terms $\mathcal{P}_2E_2$, $\mathcal{B}_1E_1$ and $\mathcal{B}_2E_2$ can be analyzed in a similar way.
\end{proof}


\begin{lemma}\label{lem:interpolation-error-2}
Let Assumptions \ref{assumption:1} and \ref{assumption-2} hold true.
Let $E^I_{12}$ denote the Lagrange interpolant of $
E_{12}$ on the Bakhvalov-type mesh $\mathcal{T}_N$. Then there exists a constant $C$ such
that the following interpolation error estimates hold true:  
\begin{align*}
&\Vert E_{12}-E^I_{12} \Vert_{\Omega\setminus \tau_{N/2-1,N/2-1} }\le C \varepsilon N^{-k}+C\varepsilon^{1/2}N^{-(k+1)}, \\ 
&\Vert E_{12}-E^I_{12} \Vert \le C \varepsilon N^{-k}+C\varepsilon^{1/2}N^{-(k+1)}+CN^{-1-2\sigma},\\ 
&\Vert E_{12}-E^I_{12} \Vert_{\varepsilon} \le C N^{-k},  \\
&\Vert  \mathcal{P}_{12} E_{12}   \Vert_{\varepsilon} \le CN^{-1/2-2\sigma}, 
\end{align*}
where $\mathcal{P}_{12} E_{12}$ is defined in \eqref{eq:P12E12}.
\end{lemma}
\begin{proof}
Consider 
\begin{equation}\label{eq:C0}
\begin{aligned}
\Vert E_{12}-E_{12}^I \Vert^2
=&\Vert E_{12}-E_{12}^I \Vert^2_{\Omega_{11}}+\Vert E_{12}-E_{12}^I \Vert^2_{\Omega_{12}\cup\Omega_{21}\cup (\Omega_{22}\setminus \tau_{N/2-1,N/2-1}) }\\
&+\Vert E_{12}-E_{12}^I \Vert^2_{ \tau_{N/2-1,N/2-1} }.
\end{aligned}
\end{equation}

Similar to \eqref{eq:A1}, the estimation for $ \Vert E_{12}-E_{12}^I \Vert^2_{\Omega_{11}} $ is as follows:
\begin{equation}\label{eq:C1}
 \Vert E_{12}-E_{12}^I \Vert^2_{\Omega_{11}} 
\le    C  \varepsilon^2 N^{-2k}.
\end{equation}

To analyze the second term in \eqref{eq:C0}, we frequently use the following estimation
\begin{equation}\label{eq:CC}
\Vert E_{12}^I \Vert^2_{\tau_{i,j}}\le 
\Vert E_{12}^I \Vert^2_{ \infty,\tau_{i,j} }h_{i,x}h_{j,y}
\le
Ce^{-2(\beta_1 x_i+\beta_2y_j)/\varepsilon} h_{i,x}h_{j,y}.
\end{equation}
From \eqref{eq:(2.1e)}, direct calculations yield 
\begin{equation}\label{eq:C6}
\Vert E_{12}  \Vert_{\Omega\setminus \Omega_{11}}\le C \varepsilon N^{-\sigma}.
\end{equation}
From \eqref{eq:CC} and Lemma \ref{lem:anisotropic-interpolation}, one has 
\begin{equation}\label{eq:C2}
\begin{aligned}
\Vert E_{12}^I \Vert^2_{\Omega_{12} }=&\sum_{i=N/2-1}^{N-1}\sum_{j=0}^{N/2-2} \Vert E_{12}^I \Vert^2_{\tau_{i,j}} 
\le   C \sum_{i=N/2-1}^{N-1}\sum_{j=0}^{N/2-2} e^{-2(\beta_1 x_i+\beta_2y_j)/\varepsilon} h_{i,x}h_{j,y}\\
\le &
C \sum_{i=N/2-1}^{N-1} e^{-2 \beta_1 x_i/\varepsilon} h_{i,x}\;\sum_{j=0}^{N/2-2} e^{-2\beta_2y_j/\varepsilon} h_{j,y}
\le
C  \varepsilon N^{-2\sigma},
\end{aligned}
\end{equation}
where Lemma \ref{lem:Bakhvlov-mesh}  has been used.
Similarly, we have
\begin{equation}\label{eq:C3}
\Vert E_{12}^I \Vert^2_{\Omega_{21} } \le   C  \varepsilon N^{-2\sigma}.
\end{equation}
 Using \eqref{eq:CC} again, we have
\begin{equation}\label{eq:C4}
\begin{aligned}
&\Vert E_{12}^I \Vert^2_{ (\Omega_{22}\setminus \tau_{N/2-1,N/2-1}) }=
\sum_{j=N/2}^{N-1} \Vert E_{12}^I \Vert^2_{\tau_{N/2-1,j}}+
\sum_{i=N/2}^{N-1}\sum_{j=N/2-1}^{N-1} \Vert E_{12}^I \Vert^2_{\tau_{i,j}} \\
\le &
\sum_{j=N/2}^{N-1} e^{-2(\beta_1 x_{N/2-1}+\beta_2y_j)/\varepsilon} h_{N/2-1,x}h_{j,y}+
\sum_{i=N/2}^{N-1}\sum_{j=N/2-1}^{N-1} e^{-2(\beta_1 x_i+\beta_2y_j)/\varepsilon} h_{i,x}h_{j,y}\\
\le & C  N^{-1-2\sigma}\varepsilon^{2\sigma}+C N^{-2\sigma}\varepsilon^{2\sigma}
\end{aligned}
\end{equation}
and
\begin{equation}\label{eq:C5}
\Vert E_{12}^I \Vert^2_{ \tau_{N/2-1,N/2-1} } \le   C   N^{-2-4\sigma}.
\end{equation}
Then the triangle inequality and \eqref{eq:C1}, \eqref{eq:C6}--\eqref{eq:C5} yield the first and the second bounds.

Now we consider $\Vert (E_{12}-E^I_{12})_x \Vert$.  By similar  derivations for \eqref{eq:B1}, we have
\begin{equation}\label{eq:D1}
\Vert (E_{12}-E^I_{12})_x \Vert_{\Omega_{11}}\le C N^{-k}.
\end{equation}
Similar to  \eqref{eq:B2}, the following bound can be obtained
\begin{equation}\label{eq:D2}
\Vert (E_{12}-E^I_{12})_x \Vert_{\Omega\setminus\Omega_{11}}\le C \varepsilon^{-1/2}N^{1-\sigma}.
\end{equation}
 Combing \eqref{eq:D1} and \eqref{eq:D2}, we prove
$$
\Vert (E_{12}-E^I_{12})_x \Vert\le C \varepsilon^{-1/2}N^{1-\sigma}
$$
and the same bound for $\Vert (E_{12}-E^I_{12})_y \Vert$. Thus we prove the third bound. Similar to   \eqref{eq:B5}, the  final bound can be proved.
\end{proof}

\begin{lemma}\label{lem:interpolation-error-3}
Let Assumptions \ref{assumption:1} and \ref{assumption-2} hold true.
Let $S^I$ and $u^I$ denote the Lagrange interpolants of $
S$ and $u$ on the Bakhvalov-type mesh $\mathcal{T}_N$, respectively. Let $\Pi u $ and $\pi_i E_i$, $i=1,2,12$, be defined in \eqref{eq:interpolant-u} and \eqref{eq:interpolant-E-1}, respectively. Then there exists a constant $C$ such
that the following interpolation error estimates hold true:  
\begin{align*}
&\sum_{i=1,2,12}\Vert \pi_i E_i-E_i \Vert \le CN^{-(k+1)},\\
&\Vert \nabla (S-S^I) \Vert +\Vert u-u^I \Vert_{\varepsilon}+\Vert u-\Pi u \Vert_{\varepsilon} \le C N^{-k}.
\end{align*}

\end{lemma}
\begin{proof}
Check the derivations in \eqref{eq:B5} and one finds that $ \Vert \mathcal{P}_1 E_1 \Vert\le C N^{-1/2-\sigma}$. Thus 
\begin{equation*} 
\Vert \pi_1 E_1-E_1 \Vert\le \Vert   E_1^I-E_1 \Vert +\Vert \mathcal{P}_1 E_1 \Vert\le CN^{-(k+1)}.
\end{equation*}
Similarly, we can prove the bounds for $E_2$ and $E_{12}$.

From  Lemma   \ref{lem:anisotropic-interpolation} and \eqref{eq:(2.1b)}, we   prove $\Vert \nabla (S-S^I) \Vert\le C N^{-k}$ and $\Vert S-S^I \Vert_{\varepsilon}\le C(\varepsilon^{1/2}+N^{-1})N^{-k}$ easily. Then \eqref{eq:decomposition}, the triangle inequality, Lemmas \ref{lem:interpolation-error-1} and \ref{lem:interpolation-error-2} yield
$\Vert u-u^I \Vert_{\varepsilon} \le C N^{-k}$. Besides, from \eqref{eq:pi-E-0}, Lemmas \ref{lem:interpolation-error-1} and \ref{lem:interpolation-error-2} one has 
$\Vert u-\Pi u \Vert_{\varepsilon}\le C N^{-k}$.

\end{proof}

\section{Uniform convergence}
Set $\chi:=\Pi u-u^N$.
From \eqref{eq:coercivity}, \eqref{eq:Galerkin orthogonality}, \eqref{eq:decomposition}, \eqref{eq:interpolant-u}  and integration by parts, one has
\begin{equation}\label{eq:uniform-convergence-1}
\begin{split}
&\alpha \Vert \chi \Vert_{\varepsilon}^2\le a(\chi,\chi)=a(\Pi u-u,\chi)\\
=&\varepsilon \int_{\Omega} \nabla (\Pi u-u) \nabla \chi \mr{d}x\mr{d}y+\sum_{i=1,2}\int_{\Omega} ( E^I_i-\mathcal{P}_i E_i-E_i)\; \b{b}\cdot\nabla\chi\mr{d}x\mr{d}y\\
&+\int_{\Omega} ( E^I_{12}-\mathcal{P}_{12} E_{12}-E_{12}) \; \b{b}\cdot\nabla\chi\mr{d}x\mr{d}y-\int_{\Omega}\b{b}\cdot\nabla (S^I-S) \; \chi\mr{d}x\mr{d}y\\
&+\sum_{i=1,2,12} \int_{\Omega}(\nabla \cdot \b{b})  (\pi_i E_i-E_i) \;\chi\mr{d}x\mr{d}y+\int_{\Omega}c (\Pi u-u)\chi\mr{d}x\mr{d}y\\
&+\sum_{i=1,2} \int_{\Omega} ( -\mathcal{B}_i E_i)\; \b{b}\cdot\nabla\chi\mr{d}x\mr{d}y\\
=:&\mr{I}+\mr{II}+\mr{III}+\mr{IV}+\mr{V}+\mr{VI}+\mr{VII}.
\end{split}
\end{equation}
Now we  analyze the terms on the right-hand side of \eqref{eq:uniform-convergence-1}. The Cauchy-Schwarz inequality and Lemma \ref{lem:interpolation-error-3} yield 
\begin{equation}\label{eq:I}
\begin{aligned}
&( \mr{I}+\mr{VI}) +( \mr{IV}+\mr{V})\\
\le & C \Vert \Pi u-u \Vert_{\varepsilon}\Vert \chi \Vert_{\varepsilon} 
+C \big( \Vert \nabla (S^I-S) \Vert+\sum_{i=1,2,12} \Vert \pi_i E_i-E_i \Vert \big)\Vert \chi \Vert \\
\le & CN^{-k}\Vert \chi \Vert_{\varepsilon}.
\end{aligned}
\end{equation}

We put the arguments for $\mr{II}$, $\mr{III}$  and $\mr{VII}$  in the following three lemmas.   
\begin{lemma}\label{lem:II}
Let Assumptions \ref{assumption:1} and \ref{assumption-2} hold true. Let $\pi_i E_i$ with $i=1,2$ be defined in \eqref{eq:interpolant-E-1}. Then one has
\begin{equation}\label{eq:II}
|\mr{II}|=\left|\sum_{i=1,2}\int_{\Omega} (E^I_i-\mathcal{P}_i E_i-E_i)\; \b{b}\cdot\nabla\chi\mr{d}x\mr{d}y \right|\le CN^{-(k+1/2)} \Vert \chi \Vert_{\varepsilon}.
\end{equation}
\end{lemma}
\begin{proof}
For the term  $\mr{II}$, we just consider $E_1$ since we can analyze $E_2$ in a similar way. 
Set $D_0:= [x_{N/2-1},x_{N/2}]\times [0,1] $. According to \eqref{eq:interpolant-E-1} and   \eqref{eq:P1E1}, one has
\begin{equation}\label{eq:II-1}
\begin{aligned}
 &\int_{\Omega} (E^I_1-\mathcal{P}_1 E_1)\; \b{b}\cdot\nabla\chi\mr{d}x\mr{d}y \\
 =& 
 \left( E^I_1-\mathcal{P}_1 E_1-E_1,\b{b}\cdot\nabla\chi\right)_{\Omega\setminus  D_0}  +
  \left( E^I_1-\mathcal{P}_1 E_1-E_1,\b{b}\cdot\nabla\chi\right)_{ D_0 } \\
=&
 \left(   E_1^I-E_1,\b{b}\cdot\nabla\chi\right)_{\Omega\setminus D_0 } 
+\left( -\mathcal{F}_1   ,\b{b}\cdot\nabla\chi\right)_{\Omega\setminus D_0 }\\
 &+ 
 \left( \mathcal{F}_2-E_1,\b{b}\cdot\nabla\chi\right)_{ D_0 }\\
 =:&\mr{T}_1+\mr{T}_2+\mr{T}_3,
\end{aligned}
\end{equation}
where $(E^I_1-\mathcal{P}_1 E_1)|_{\Omega\setminus D_0 } =(E_1^I-\mathcal{F}_1)|_{\Omega\setminus D_0 } $, $(E^I_1-\mathcal{P}_1 E_1)|_{ D_0 }=\mathcal{F}_2|_{ D_0 }$ and
\begin{align*}
\mathcal{F}_1:=&  \sum_{j=0}^{N-1}\sum_{t=0}^{k-1} E_1(x_{N/2-1},y_j^t)\theta_{N/2-1,j}^{0,t}+E_1(x_{N/2-1},y_N)\theta_{N/2-1,N}^{0,0},\\
\mathcal{F}_2:=& \sum_{j=0}^{N-1}\sum_{t=0}^{k-1} E_1(x_{N/2},y_j^t)\theta_{N/2,j}^{0,t}+E_1(x_{N/2},y_N)\theta_{N/2,N}^{0,0}.
\end{align*}

From Lemma \ref{lem:interpolation-error-1} and the Cauchy-Schwarz inequality, one has
\begin{equation}\label{eq:II-2}
\begin{aligned}
|\mr{T}_1|\le &C \Vert E_1-E^I_1 \Vert_{\Omega \setminus D_0 } \Vert \nabla\chi \Vert_{\Omega \setminus D_0 } 
\le  C \varepsilon^{1/2}N^{-(k+1/2)} \Vert \nabla\chi \Vert
\le C N^{-(k+1/2)} \Vert  \chi \Vert_{\varepsilon}.
\end{aligned}
\end{equation}

Note $|E_1(x_{N/2-1},y_j^t)|\le CN^{-\sigma}$ for any $j,t$ and $h_{N/2-2,x}\le C \varepsilon$. The Cauchy-Schwarz inequality yields
\begin{equation}\label{eq:II-3}
\begin{aligned} 
|\mr{T}_2|\le &C   \Vert \mathcal{F}_1  \Vert_{D_1} \Vert \nabla \chi \Vert_{D_1}\\
\le &
C N^{-\sigma} \Vert \nabla \chi \Vert_{D_1} \left( \sum_{j=0}^{N-1}\sum_{t=0}^{k-1} \Vert  \theta_{N/2-1,j}^{0,t} \Vert_{D_1}+\Vert  \theta_{N/2-1,N}^{0,0} \Vert_{D_1}  \right)\\
\le & 
C N^{-\sigma} \Vert \nabla \chi \Vert_{D_1} \sum_{j=0}^{N-1} h^{1/2}_{N/2-2,x}h^{1/2}_{j,y}\\
\le &
C N^{1/2-\sigma}\Vert \chi \Vert_{\varepsilon},
\end{aligned}
\end{equation}
where $D_1:=[x_{N/2-2},x_{N/2-1}]\times [0,1]$ and we have made use of the supports of hat functions $\theta^{s,t}_{i,j}$.

Now we deal with the term $\mr{T}_3$. Note $|E_1(x_{N/2},y_j^t)|\le C\varepsilon^{\sigma}$ for any $j,t$ and $h_{N/2-1,x}\le C N^{-1}$. The Cauchy-Schwarz inequality yields
\begin{equation}\label{eq:II-4}
\begin{aligned}
|\mr{T}_3|\le &C  (\Vert E_1  \Vert_{D_0}+\Vert \mathcal{F}_2  \Vert_{D_0} ) \Vert \nabla \chi  \Vert_{D_0}\\
\le & C
 \left(\varepsilon^{1/2}N^{-\sigma}+\varepsilon^{\sigma} \sum_{j=0}^{N-1} h^{1/2}_{N/2-1,x}h^{1/2}_{j,y}
\right)
\Vert \nabla \chi \Vert_{D_2}\\
\le &
C    (N^{-\sigma}+\varepsilon^{\sigma-1/2})\Vert  \chi \Vert_{\varepsilon}
\end{aligned}
\end{equation}
where direct calculations and \eqref{eq:(2.1c)}  yield $\Vert E_1  \Vert_{D_0}\le C\varepsilon^{1/2}N^{-\sigma}$  and $\Vert \mathcal{F}_2  \Vert_{D_0}$ is analyzed in a similar way to $\Vert \mathcal{F}_1  \Vert_{D_1}$ in \eqref{eq:II-3}.

Substituting \eqref{eq:II-2}--\eqref{eq:II-4} into \eqref{eq:II-1}, we are done. 
\end{proof}

\begin{lemma}\label{lem:III}
Let Assumptions \ref{assumption:1} and \ref{assumption-2} hold true. Let $\pi_{12} E_{12}$  be defined in \eqref{eq:interpolant-E-1}. Then one has
\begin{equation}\label{eq:III}
|\mr{III}|=\left| \int_{\Omega} (\pi_{12} E_{12}-E_{12})\; \b{b}\cdot\nabla\chi\mr{d}x\mr{d}y \right|\le CN^{-(k+1/2)} \Vert \chi \Vert_{\varepsilon}.
\end{equation}
\end{lemma}
\begin{proof}
Set $D_3=:\Omega\setminus \tau_{N/2-1,N/2-1}$. According to \eqref{eq:interpolant-E-1}, one has
\begin{equation}\label{eq:III-1}
\begin{aligned}
 &\int_{\Omega} (\pi_{12} E_{12}-E_{12})\; \b{b}\cdot\nabla\chi\mr{d}x\mr{d}y \\
 =& 
 \left( \pi_{12} E_{12}-E_{12},\b{b}\cdot\nabla\chi\right)_{ D_3 }
 +
  \left(  \pi_{12} E_{12}-E_{12},\b{b}\cdot\nabla\chi\right)_{\tau_{N/2-1,N/2-1} } \\
=&
 \left( E_{12}^I-E_{12},\b{b}\cdot\nabla\chi\right)_{D_3 }
 -\left(\mathcal{F}_3,\b{b}\cdot\nabla\chi\right)_{ D_3 } +
  \left(  \mathcal{F}_4-E_{12},\b{b}\cdot\nabla\chi\right)_{\tau_{N/2-1,N/2-1} } 
\\
 =:&\mr{Q}_1+\mr{Q}_2+\mr{Q}_3,
\end{aligned}
\end{equation}
where  $(\pi_{12} E_{12})|_{ D_3 } =(E_{12}^I-\mathcal{F}_3)|_{D_3 } $, $(\pi_{12} E_{12})|_{ \tau_{N/2-1,N/2-1} }=\mathcal{F}_4|_{ \tau_{N/2-1,N/2-1} }$ and
\begin{align*}
\mathcal{F}_3:=&   \sum_{t=1}^{k-1} E_{12}(x_{N/2-1},y_{N/2-1}^t)\theta_{N/2-1,N/2-1}^{0,t}+E_{12}(x_{N/2-1},y_{N/2-1})\theta_{N/2-1,N/2-1}^{0,0}\\
&+\sum_{s=1}^{k-1} E_{12}(x_{N/2-1}^s,y_{N/2-1})\theta_{N/2-1,N/2-1}^{s,0},\\
\mathcal{F}_4:=&  \sum_{t=0}^{k-1} E_{12}(x_{N/2},y_{N/2-1}^t)\theta_{N/2,N/2-1}^{0,t}
+E_{12}(x_{N/2},y_{N/2})\theta_{N/2,N/2}^{0,0}\\
&+\sum_{s=0}^{k-1} E_{12}(x_{N/2-1}^s,y_{N/2})\theta_{N/2-1,N/2}^{s,0}.
\end{align*}

 From Lemma \ref{lem:interpolation-error-2} and the Cauchy-Schwarz inequality, one has
\begin{equation}\label{eq:III-2}
\begin{aligned}
|\mr{Q}_1|\le C \Vert  E_{12}^I-E_{12} \Vert_{ D_3 } \Vert \nabla\chi \Vert_{ D_3 } 
\le  C \varepsilon^{1/2}N^{-(k+1/2)} \Vert \nabla\chi \Vert
\le C N^{-(k+1/2)} \Vert  \chi \Vert_{\varepsilon}.
\end{aligned}
\end{equation}

Note $|E_{12}(x_{N/2-1}^s,y_{N/2-1}^t)|\le CN^{-2\sigma}$ for any $0\le s,t\le k-1$
and $h_{N/2-2,x},h_{N/2-2,y}\le C\varepsilon$. Then one has
\begin{align*}
\Vert \mathcal{F}_3 \Vert_{D_3}\le &
CN^{-2\sigma}\left( \sum_{t=1}^{k-1} \Vert \theta_{N/2-1,N/2-1}^{0,t} \Vert_{D_3}+\Vert  \theta_{N/2-1,N/2-1}^{0,0} \Vert_{D_3}
+\sum_{s=1}^{k-1} \Vert \theta_{N/2-1,N/2-1}^{s,0} \Vert_{D_3}  \right)\\
\le &
CN^{-2\sigma}\left( h_{N/2-1,x}^{1/2}h_{N/2-2,y}^{1/2}+h_{N/2-2,x}^{1/2}h_{N/2-2,y}^{1/2} 
+h_{N/2-2,x}^{1/2}h_{N/2-1,y}^{1/2}  \right)\\
\le &
C\varepsilon^{1/2}N^{-1/2-2\sigma}.
\end{align*}
Thus the Cauchy-Schwarz inequality yields
\begin{equation}\label{eq:III-3}
\begin{aligned} 
&|\mr{Q}_2|\le C   \Vert \mathcal{F}_3 \Vert_{D_3} \Vert \nabla \chi \Vert_{D_3}
\le
C N^{-1/2-2\sigma}\Vert \chi \Vert_{\varepsilon}.
\end{aligned}
\end{equation}

 Note $|E_{12}(x_{N/2}^s,y_{N/2-1}^t)|+|E_{12}(x_{N/2-1}^s,y_{N/2}^t)|+|E_{12}(x_{N/2},y_{N/2})|\le C\varepsilon^{\sigma}N^{-\sigma}$ for any $0\le s,t \le k-1$. The Cauchy-Schwarz inequality yields
\begin{equation}\label{eq:III-4}
\begin{aligned}
|\mr{Q}_3|\le &C  (\Vert \mathcal{F}_4  \Vert_{\tau_{N/2-1,N/2-1}}
+\Vert E_{12}  \Vert_{ \tau_{N/2-1,N/2-1} } ) \Vert \nabla \chi  \Vert_{ \tau_{N/2-1,N/2-1} }\\
\le & C
 \left(\varepsilon^{\sigma}N^{-\sigma}h^{1/2}_{N/2-1,x}h^{1/2}_{N/2-1,y}+\varepsilon N^{-2\sigma} 
\right)
\Vert \nabla \chi \Vert_{ \tau_{N/2-1,N/2-1} }\\
\le &
C    \varepsilon^{1/2} N^{-2\sigma} \Vert  \chi \Vert_{\varepsilon}.
\end{aligned}
\end{equation}

Substituting \eqref{eq:III-2}--\eqref{eq:III-4} into \eqref{eq:III-1}, we are done. 
\end{proof}

\begin{lemma}\label{lem:VII}
Let Assumptions \ref{assumption:1} and \ref{assumption-2} hold true. Let $\mathcal{B}_i E_i$ with $i=1,2$ be defined in \eqref{eq:interpolant-E-1}. Then one has
\begin{equation}\label{eq:VII}
|\mr{VII}|=\left|\sum_{i=1,2} \int_{\Omega} ( -\mathcal{B}_i E_i)\; \b{b}\cdot\nabla\chi\mr{d}x\mr{d}y \right|\le CN^{-k}  R(N,\varepsilon)\Vert \chi \Vert_{\varepsilon},
\end{equation}
where $R(N,\varepsilon)=N^{-3/2}| \ln(\varepsilon N) |^{1/2}$.
\end{lemma}
\begin{proof}
We just present the analysis on the term involved with $\mathcal{B}_1E_1$ and the other could be analyzed in a similar way.  H\"{o}lder inequalities 
\begin{equation*}
\begin{aligned}
&\left|  \int_{\Omega} ( -\mathcal{B}_1 E_1)\; \b{b}\cdot\nabla\chi\mr{d}x\mr{d}y \right|\\
\le &\sum_{j=0,N-1}\Vert \mathcal{B}_1 E_1 \Vert_{ \infty,\tau_{N/2-1,j} }  \Vert \b{b}\cdot\nabla\chi \Vert_{ 1,\tau_{N/2-1,j} }
\\
\le &C \sum_{j=0,N-1} \left( N^{-\sigma} h_{N/2-1,x}^{1/2} h_{j,y}^{1/2} \Vert \b{b}\cdot\nabla\chi \Vert_{ \tau_{N/2-1,j} } \right)\\
\le &
C (N^{-(1+\sigma)}+N^{-(1/2+\sigma)} | \ln(\varepsilon N) |^{1/2} ) \Vert \chi \Vert_{\varepsilon }\\
\le &
C N^{-(1/2+\sigma)} | \ln(\varepsilon N) |^{1/2} \Vert \chi \Vert_{\varepsilon },
\end{aligned}
\end{equation*}
where we have used $h_{0,y}\le C \varepsilon N^{-1}$ in Lemma \ref {lem:Bakhvlov-mesh} and $h_{N/2-1,x}=-\frac{\sigma \varepsilon}{\beta} \ln (\varepsilon N)$ from \eqref{B-mesh-1}.

\end{proof}
\begin{remark}
Practically
$R(N,\varepsilon)$ is bounded: If we assume $N\ge 10$ and $\varepsilon \ge 10^{-1001}$,
then
$$
R(N,\varepsilon)\le \sqrt{\ln 10}.
$$
\end{remark}

Now we are in a position to present  the main result.  
\begin{theorem}\label{the:main result}
Let Assumptions \ref{assumption:1} and \ref{assumption-2} hold true. Let $u$ and $u^N$ be the solutions of \eqref{eq:model problem} and \eqref{eq: FE}, respectively.  Then one has
\begin{align*}
 \Vert u-u^N \Vert_{\varepsilon}\le & CN^{-k}R(N,\varepsilon),
\end{align*}
where $R(N,\varepsilon)=N^{-3/2}| \ln(\varepsilon N) |^{1/2}$.
\end{theorem}
\begin{proof}
Substituting \eqref{eq:I}, \eqref{eq:II}, \eqref{eq:III} and \eqref{eq:VII}  into \eqref{eq:uniform-convergence-1}, we obtain   $\Vert \Pi u-u^N \Vert_{\varepsilon}\le C N^{-k}R(N,\varepsilon)$.
From a triangle  inequality and  Lemma \ref{lem:interpolation-error-3}, one has
$$
\Vert u-u^N \Vert_{\varepsilon}
\le   \Vert u-\Pi u \Vert_{\varepsilon}+\Vert \Pi u-u^N \Vert_{\varepsilon}\le  
 C N^{-k}R(N,\varepsilon).
$$
Thus we are done.
\end{proof}

\begin{remark}
In Theorem \ref{the:main result}, we show the almost optimal estimation, which is slightly better than the bound in \cite[Theorem 1]{Roos1Scho2:2012-Analysis}.
\end{remark}
 
%
%
%
\section{Numerical Experiments}\label{sec:numerical experiments}
In this section we present numerical experiments that support our theoretical results.  All calculations were carried out using Intel Visual Fortran 11 and the discrete problems
were solved by the nonsymmetric iterative solver GMRES; see, e.g., \cite{Ben1Gol2Lie3:2005-Numerical}.

For our experiments we used the boundary value problem
\begin{align*}
-\varepsilon\Delta u-(2+2x-y)u_{x}-(3-x+2y)u_{y}+u&=f(x,y)\quad&&\text{in $\Omega=(0,1)^{2}$},\\
u&=0 \quad&&\text{on $\partial\Omega$,}
\end{align*}
where the right-hand side  $f$ is chosen such that
\begin{equation*}
u(x,y)=2\sin (\pi x)\left(1-e^{-\frac{2x}{\varepsilon} } \right)
(1-y)^{2} \left( 1-e^{-\frac{y}{\varepsilon} } \right)
\end{equation*}
is the exact solution. This solution exhibits typical exponential layer behaviour as described in Assumption~\ref{assumption:1}.

Numerical results are presented in Tables \ref{table:alex-1} and \ref{table:alex-2} and Figure  \ref{fig:comparison}, which support our  main result. Tables \ref{table:alex-1} and \ref{table:alex-2}  list   errors in the energy norm, i.e., $\Vert u-u^N \Vert_{\varepsilon}$,  for $\varepsilon=10^{-4},10^{-5},\ldots, 10^{-8}$ and $N=8,16,32,64,128,256$, in the cases of $k=1, 2$. These data show  uniform convergence with respect to the singular perturbation parameter $\varepsilon$. 
In the cases of  $k=3,4$, errors and convergence orders also 
show  uniform convergence, which are  plotted in Figure \ref{fig:comparison}.

Here we also compare two meshes, which are Bakhvalov-type mesh \eqref{B-mesh-1} (denoted by B-mesh) and  Bakhvalov-Shishkin mesh (denoted by B-S-mesh).    B-S-mesh    \cite{Linb:2000-Analysis}  yields an optimal convergence order for bilinear element and has excellent performances. According to  \cite{Linb:2000-Analysis},
B-S-mesh is defined by tensor product and  the mesh points in $x-$direction are defined by
\begin{equation*}
\varphi(i/N)=
\left\{
\begin{split}
&-\frac{(k+1)\varepsilon}{\beta_1} \ln \left( \frac{N^2-2i(N-1)}{N^2}\right)\quad &&\text{for $i=0,\ldots,N/2$},\\
&\varphi(1/2)+2(1-\varphi(1/2))/(i/N-1/2)\quad &&\text{for $i=N/2+1,\ldots,N$}.
\end{split}
\right.
\end{equation*}
The mesh points along $y-$direction can be  defined similarly. 

%
%

To compare B-mesh and B-S-mesh, we plot
 errors in the energy norm on these meshes for $\varepsilon=10^{-8}$ and $N=8,16,32,64,128$, for the cases $k=1,2,3,4$  on log-log chart in Figure  \ref{fig:comparison}. We can see that these meshes have similar performances  and B-S-mesh  yields slightly smaller errors.

Besides, from numerical experiments we find that the linear systems become harder to be solved by iterative solvers when $k$ and $N$ become  bigger and $\varepsilon$ becomes smaller.

\begin{table}[h]
\caption{Errors and orders  in the energy norm on B-mesh  for $k=1$}
\footnotesize
\begin{tabular*}{\textwidth}{@{\extracolsep{\fill}} c | cccccc  }
\hline
 \diagbox{$\varepsilon$}{$N$}   & 8 & 16 & 32 & 64 &128&256 \\
  \hline
 $10^{-4}$ & 0.227E+0   & 0.109E+0    & 0.540E-1& 0.269E-1&0.135E-1&0.673E-2  \\
             $10^{-4}$ & 1.06   & 1.02& 1.00&1.00&1.00         &---   \\
            $10^{-5}$ & 0.228E+0   & 0.109E+0    & 0.540E-1& 0.269E-1&0.135E-1&0.673E-2  \\
             $10^{-5}$ & 1.06   & 1.02& 1.00&1.00&1.00          &---\\
            $10^{-6}$ & 0.229E+0   & 0.109E+0    & 0.540E-1& 0.269E-1&0.135E-1&0.673E-2  \\
             $10^{-6}$ & 1.07   & 1.02& 1.00&1.00&1.00         &---\\
            $10^{-7}$ & 0.231E+0   & 0.110E+0    & 0.541E-1& 0.269E-1&0.135E-1&0.673E-2  \\
             $10^{-7}$ & 1.08  & 1.02& 1.00&1.00&1.00          &---\\
            $10^{-8}$ & 0.234E+0   & 0.110E+0    & 0.541E-1& 0.269E-1&0.135E-1&0.673E-2  \\
             $10^{-8}$ & 1.09  & 1.02& 1.01&1.00&1.00         &---\\
\hline
\end{tabular*}
\label{table:alex-1}
\end{table}

\begin{table}[h]
\caption{Errors and orders  in the energy norm on B-mesh  for $k=2$}
\footnotesize
\begin{tabular*}{\textwidth}{@{\extracolsep{\fill}} c | ccccc    }
\hline
 \diagbox{$\varepsilon$}{$N$}    & 8 & 16 & 32 & 64 &128  \\
  \hline
 $10^{-4}$ & 0.502E-1  & 0.105E-1    & 0.247E-2& 0.603E-3&0.150E-3 \\
             $10^{-4}$ & 2.25   & 2.10& 2.03&2.01  &---     \\
 $10^{-5}$ & 0.501E-1  & 0.105E-1    & 0.247E-2& 0.604E-3&0.150E-3  \\
             $10^{-5}$ & 2.25   & 2.09& 2.03&2.01    &---     \\
 $ 10^{-6}$ & 0.500E-1  & 0.105E-1    & 0.246E-2& 0.604E-3&0.150E-3  \\
             $10^{-6}$ & 2.25   & 2.09& 2.03&2.01   &---     \\
 $10^{-7}$ & 0.500E-1  & 0.105E-1    & 0.246E-2& 0.604E-3&0.150E-3  \\
             $10^{-7}$ & 2.25   & 2.09& 2.03&2.01   &---     \\
 $10^{-8}$ & 0.502E-1  & 0.105E-1    & 0.246E-2& 0.604E-3&0.150E-3   \\
             $10^{-8}$ & 2.25   & 2.10& 2.03&2.01   &---     \\
\hline
\end{tabular*}
\label{table:alex-2}
\end{table}

  \begin{figure}[htbp]
\centering
\includegraphics[width=5in]{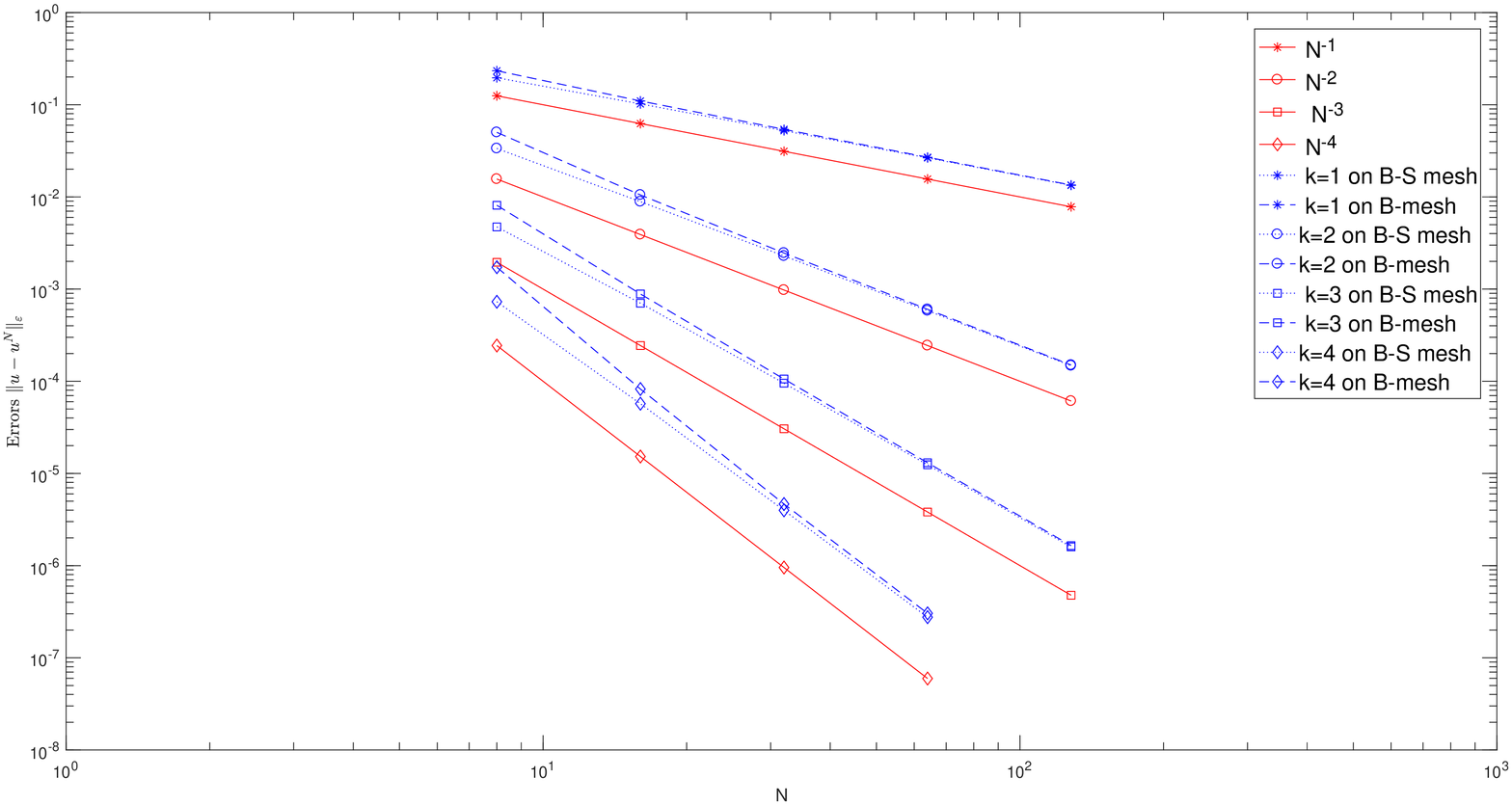}
\caption{Error:~energy norm, $\varepsilon=10^{-8}$.}
\label{fig:comparison}
\end{figure}

\section{Bibliography}
 \bibliographystyle{plain}

\end{document}